\documentclass[12pt,reqno]{amsart}

\usepackage{amssymb}
\usepackage{booktabs}
\usepackage[bookmarks=false,unicode,colorlinks,urlcolor=red,citecolor=red,linkcolor=blue]{hyperref}
\usepackage{aliascnt}
\usepackage{tikz}
\usepackage{pgfplots} 
\usepgfplotslibrary{polar}

\usepackage{color}

\usepackage[margin=1.25in]{geometry}

\newaliascnt{lemma}{theorem}

\aliascntresetthe{lemma}

\newaliascnt{proposition}{theorem}
\newtheorem{proposition}[proposition]{Proposition}
\aliascntresetthe{proposition}

\newaliascnt{corollary}{theorem}

\aliascntresetthe{corollary}

\newaliascnt{conjecture}{theorem}

\aliascntresetthe{conjecture}

\newaliascnt{example}{theorem}

\aliascntresetthe{example}

\makeatletter
\def\tagform@#1{\maketag@@@{\ignorespaces#1\unskip\@@italiccorr}}
\let\orgtheequation\theequation
\def\theequation{(\orgtheequation)}
\makeatother
\def\equationautorefname~{}

\begin{document}

\title[Torsion and ground state maxima]{Torsion and ground state maxima: \\ close but not the same}
\author[]{B. A. Benson, R. S. Laugesen, M. Minion, B. A. Siudeja}
\address{Department of Mathematics, Kansas State Univ., Manhattan,
KS 66506, U.S.A.}
\email{babenson\@@math.ksu.edu}
\address{Department of Mathematics, Univ.\ of Illinois, Urbana,
IL 61801, U.S.A.}
\email{Laugesen\@@illinois.edu}
\address{Lawrence Berkeley National Lab, Berkeley, CA 94720, U.S.A.}
\email{mlminion\@@lbl.gov}
\address{Department of Mathematics, Univ.\ of Oregon, Eugene,
OR 97403, U.S.A.}
\email{Siudeja\@@uoregon.edu}
\date{\today}


\begin{abstract}
Could the location of the maximum point for a positive solution of a semilinear Poisson equation on a convex domain be independent of the form of the nonlinearity? Cima and Derrick found certain evidence for this surprising conjecture. 

We construct counterexamples on the half-disk, by working with the torsion function and first Dirichlet eigenfunction. On an isosceles right triangle the conjecture fails again. Yet the conjecture has merit, since the maxima of the torsion function and eigenfunction are unexpectedly close together. It is an open problem to quantify this closeness in terms of the domain and the nonlinearity.   
\end{abstract}

\maketitle

\section{\bf Introduction}
\label{sec:intro}

Suppose the Poisson equation
\[
\begin{cases}
  -\Delta u = f(u) & \text{in $\Omega$,} \\
 \quad \ \ u = 0 & \text{on $\partial \Omega$,}
\end{cases}
\]
has a positive solution on the bounded convex plane domain $\Omega$. Here the nonlinearity $f$ is assumed to be Lipschitz and \emph{restoring}, which means $f(z)>0$ when $z>0$. Cima and Derrick  \cite{CD11,CDK14} have conjectured that the location of the maximum point of $u$ is independent of the form of the nonlinearlity $f$.  

This conjecture sounds impossible, since the graph of the solution must vary with the nonlinearity. Numerical computations by Cima and co-authors give surprising support for the conjecture, though, and \autoref{fig:levelcurves} provides further food for thought by considering a triangular domain and plotting the level curves and maximum point for the choices $f(z)=1$ and $f(z)=\lambda z$. The corresponding linear Poisson equations describe the \emph{torsion function} and the \emph{ground state of the Laplacian} (see below). Our solutions were computed numerically by the finite element method on a mesh with approximately $10^6$ triangles. The maximum points for the two solutions in \autoref{fig:levelcurves} appear to coincide, even though the level curves differ markedly near the boundary. 

\begin{figure}[t]
    \hspace{\fill}
\includegraphics[width=3cm]{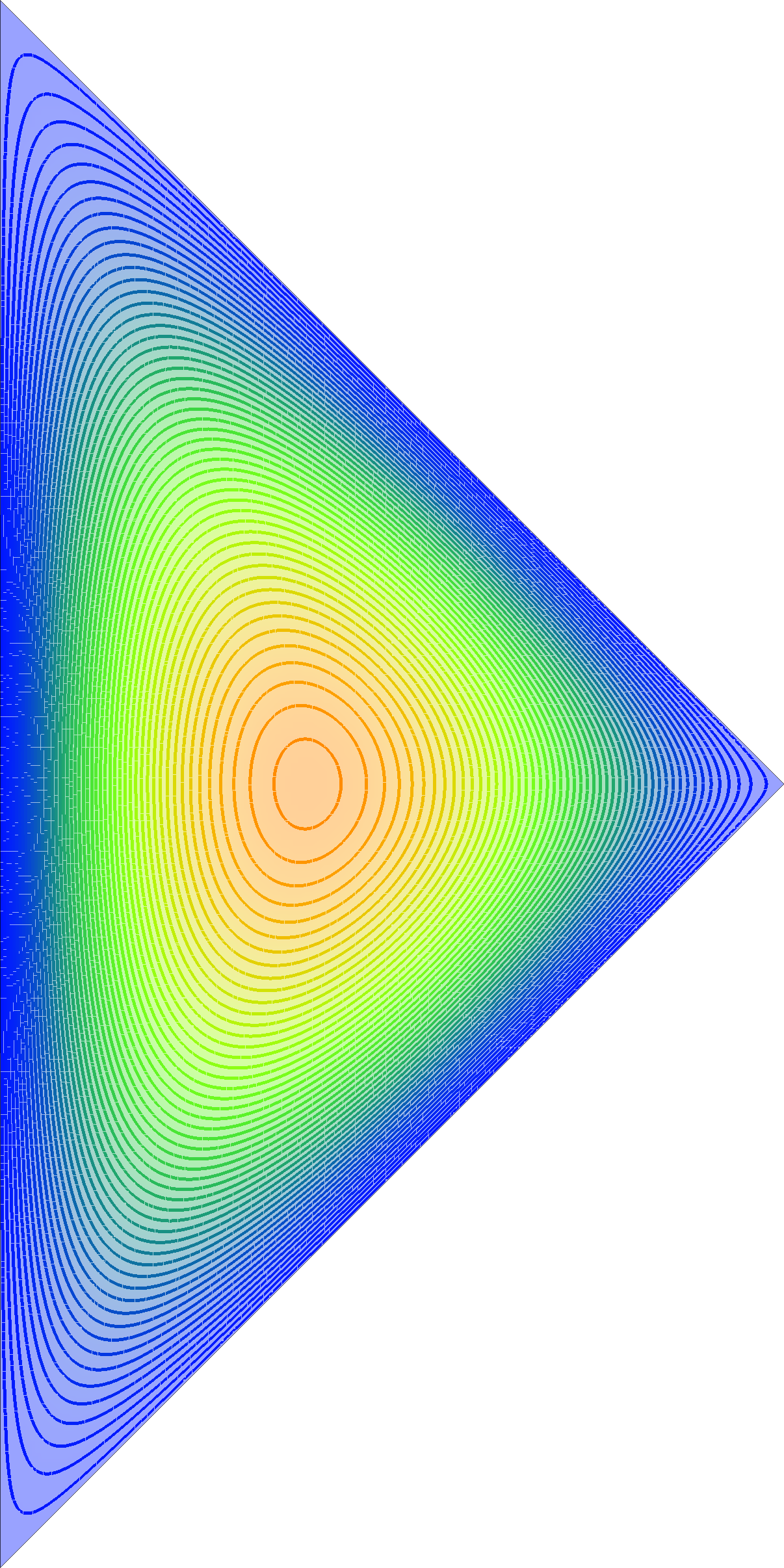}
    \hspace{\fill}
\includegraphics[width=3cm]{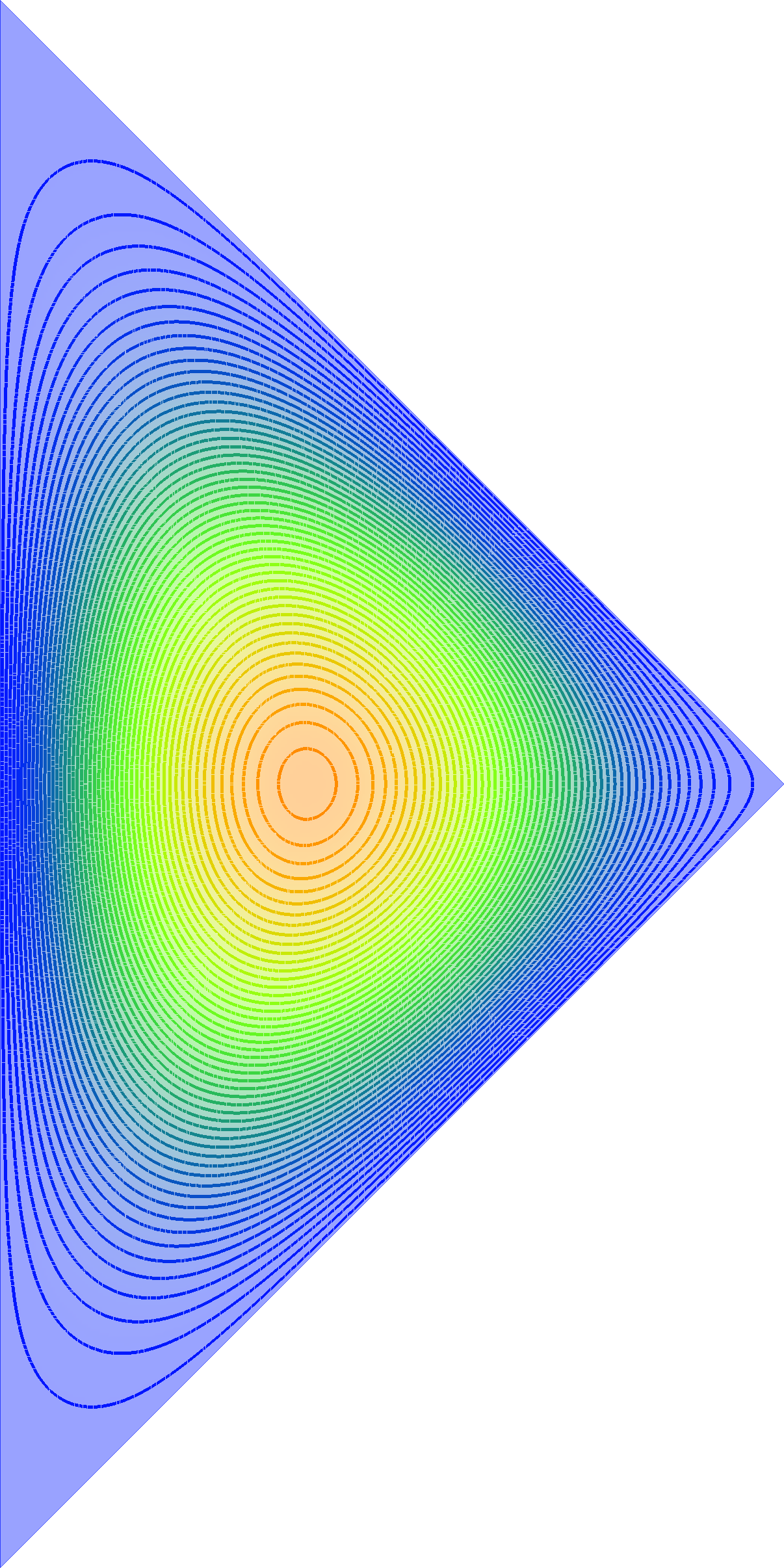}
    \hspace{\fill}
\caption{Level curves and the maximum point on a triangular domain, for solutions of two different Poisson type equations: the torsion function (left) and the first eigenfunction (right).}
\label{fig:levelcurves}
\end{figure}

We disprove the conjecture on a half-disk in \autoref{sec:half-disk}, and again on the right isosceles triangle in \autoref{sec:rightisos}. Interestingly, the conjecture is remarkably close to being true in these counterexamples, with the maximum points occurring in almost but not quite the same location. We cannot explain this unexpected closeness. 

A fascinating open problem is to bound the difference in location of the maximum points of two semilinear Poisson equations in terms of the difference between their nonlinearity functions and geometric information on the shape of the domain. Also, note that for both the half-disk and right isosceles triangle, our results show that the maximum point of the torsion function lies to the left of the maximum for the ground state (when oriented as in \autoref{fig:levelcurves}), which perhaps hints at a general principle for a class of convex domains. 

\subsection*{Notation}
The \emph{torsion} or \emph{landscape} function is the unique solution of the Poisson equation
\[
\begin{cases}
  -\Delta u = 1 & \text{in $\Omega$,} \\
 \quad \ \ u = 0 & \text{on $\partial \Omega$.}
\end{cases}
\]
Here we have chosen $f(z)=1$. Clearly $u$ is positive inside the domain, by the maximum principle. 

The \emph{Dirichlet ground state} or \emph{first Dirichlet eigenfunction of the Laplacian} is the unique positive solution of
\begin{equation*}
\begin{cases}
  -\Delta v = \lambda v& \text{in $\Omega$,} \\
\quad \ \   v = 0 & \text{on $\partial \Omega$,}
\end{cases}
\end{equation*}
where $\lambda>0$ is the first eigenvalue of the Laplacian on the domain under Dirichlet boundary conditions. Here we have chosen $f(z)=\lambda z$.

\section{\bf The half-disk}
\label{sec:half-disk}

The maximum points for the torsion function and ground state can lie so close together that one cannot distinguish them by the naked eye, as the following Proposition reveals. Yet the two points are not the same. 
\begin{proposition} \label{pr:half-disk}
Take $\Omega = \{ (x,y) : x> 0, x^2 + y^2 < 1 \}$ to be the right half-disk. On this domain the torsion function $u$ attains its maximum at $(0.48022,0)$ while the ground state $v$ attains its maximum at $(0.48051,0)$. Here the $x$-coordinates have been rounded to $5$ decimal places.
\end{proposition}
\begin{proof}
(i) The ground state is given in polar coordinates by
\[
v(r,\theta) = J_1(j_{1,1}r) \cos \theta
\]
where $J_1$ is the first Bessel function and $j_{1,1} \simeq 3.831706$ is its first positive zero. Clearly the maximum is attained on the $x$-axis, where $\theta=0$, and the function is plotted along this line in \autoref{fig:Bessel}. By setting $J_1^\prime(j_{1,1}r)=0$ and solving, we find $r = j^\prime_{1,1}/j_{1,1} \simeq 0.48051$, rounded to five decimal places, where $j^\prime_{1,1} \simeq 1.841184$ is the first zero of $J_1^\prime$.

\begin{figure}
\includegraphics[width=4cm]{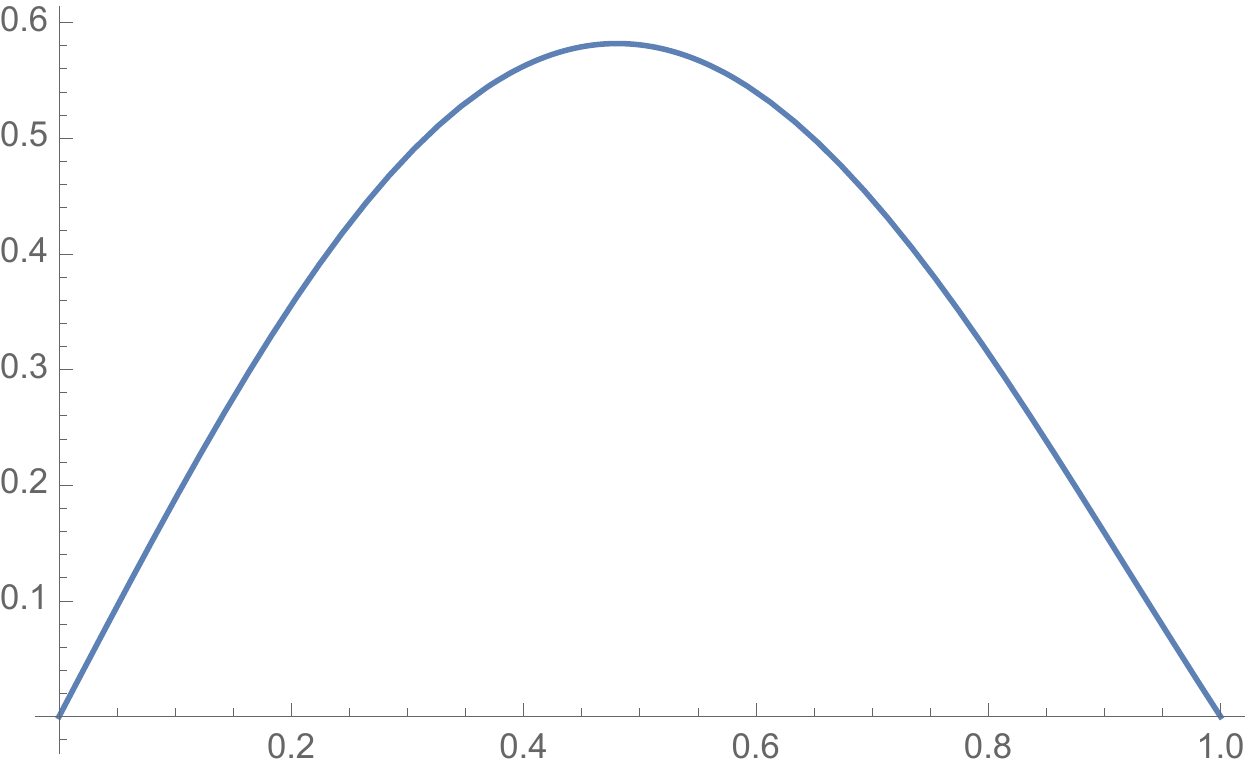}
\caption{The radial part of the ground state on the right half-disk: $v(r,0) = J_1(j_{1,1}r)$.}
\label{fig:Bessel}
\end{figure}

\medskip
(ii) The torsion function is more complicated \cite[Section 4.6.2]{W}, and is given by
\begin{align*}
u(x,y) & = \frac{1}{4\pi} \Big[ -2 \pi x^2 - 2 x \Big( (x^2 + y^2)^{-1} - 1\Big) \\
& \qquad \quad + \Big( 2 + (x^2 - y^2) \big( (x^2 + y^2)^{-2} + 1 \big) \Big) \arctan \Big( \frac{2x}{1 - (x^2 + y^2)} \Big) \\
& \qquad \quad + xy \Big( (x^2 + y^2)^{-2} - 1 \Big) \log \frac{x^2 + (1 + y)^2}{x^2 + (1 - y)^2} \, \Big] .
\end{align*}
One verifies the Dirichlet boundary condition on the right half-disk by examining four cases: (i) $u=0$ if $x=0$ and $0<|y|<1$, (ii) $u \to 0$ as $(x,y) \to (0,0)$, (iii) $u \to 0$ as $(x,y) \to (0,\pm 1)$, and (iv) $u \to 0$ as $(x,y) \to (x_1,y_1)$ with $x_1>0$ and $x_1^2+y_1^2=1$. 

To check $u$ satisfies the Poisson equation $-\Delta u=1$, a lengthy direct calculation suffices. 

We claim $u$ attains its maximum at a point on the horizontal axis. For this, first notice $u$ is even about the $x$-axis by definition, meaning $u(x,y)=u(x,-y)$. Hence the harmonic function $u_y$ equals zero on the $x$-axis for $0<x<1$. Further, $u_y \leq 0$ at points on the unit circle lying in the open first quadrant, since $u>0$ in the right half-disk and $u=0$ on the boundary. Also, one can compute that $u_y(x,y)$ approaches $0$ as $(x,y) \to (0,0)$ or $(x,y) \to (1,0)$ or $(x,y) \to (0,1)$ from within the first quadrant of the unit disk. Lastly $u_y$ vanishes on the $y$-axis for $0 < y <1$ (since $u=0$ there). Hence we conclude from the maximum principle that $u_y \leq 0$ in the first quadrant of the unit disk, and so $u$ attains its maximum somewhere on the $x$-axis. 

On the $x$-axis we have 
\[
u(x,0) = \frac{1}{4\pi} \Big[ -2 \pi x^2 - 2 x^{-1} + 2x + (2 + x^{-2} +x^2) \arctan \Big( \frac{2x}{1 - x^2} \Big) \, \Big]
\]
for $0<x<1$. Clearly $u(0,0)=u(1,0)=0$, and 
\[
u_x(x,0) = \frac{1}{\pi x^3} \big[ x + x^3 - \pi x^4 + \frac{1}{2}(x^4 - 1) \arctan \Big( \frac{2x}{1 - x^2} \Big) \big] .
\]
One can show by taking another derivative and applying elementary estimates that $u(x,0)$ is concave. Calculations show $u_x(x,0)$ is positive at $x=0.480219$ and negative at $x=0.480220$, and so the maximum of $u$ lies between these two points, that is, at $x=0.48022$ to $5$ decimal places.  
\end{proof}

\section{\bf The right isosceles triangle}
\label{sec:rightisos}

\begin{proposition} \label{pr:rightisos}
    Take $\Omega = \{ (x,y) : 0 < x < 1, |y|< 1-x \}$, which is an isosceles right triangle. On this domain the torsion function $u$ attains its maximum at $(0.39168,0)$ while the ground state $v$ attains its maximum at $(0.39183,0)$. Here the $x$-coordinates have been rounded to $5$ decimal places.
\end{proposition}
\begin{proof}
(i) Rotate the triangle by 45 degrees clockwise about the origin and scale up by a factor of $\pi/\sqrt{2}$, then translate by $\pi/2$ to the right and upwards, so that the triangle becomes 
\[
T = \{ (x,y) : 0<y<x<\pi \}. 
\]
This new triangle has ground state 
\[
v(x,y) = \sin x \sin 2y - \sin 2x \sin y = 2 \sin x \sin y (\cos y - \cos x) > 0
\]
with eigenvalue $1^2 + 2^2 =5$. One checks easily that $v=0$ on the boundary of $T$, where $y=0$ or $x=\pi$ or $y=x$. To find the maximum point, set $v_x=0$ and $v_y=0$ and deduce $\cos 2x = \cos x \cos y = \cos 2y$. Therefore the maximum lies on the line of symmetry $y=\pi-x$ of the triangle $T$. A little calculus shows that $v(x,\pi-x)$ attains its maximum when $x=\arcsin(1/\sqrt{3})+\pi/2$. Hence the ground state of the original triangle attains its maximum at $\big( (2/\pi) \arcsin(1/\sqrt{3}),0 \big) = (0.39183,0)$ to $5$ decimal places. 

\medskip
(ii) The torsion function on the triangle $T$ is 
\begin{align*}
& \! \! \! u(x,y) \\
& = - \frac{1}{4}(x-y)^2 + \sum_{n=1}^\infty \frac{n^2 \pi^2 -2\big( 1 - (-1)^n \big)}{2\pi n^3 \sinh n\pi} \Big[ \sinh nx \sin ny - \sin nx \sinh ny \\
& \qquad \qquad \qquad + \sin n(\pi-x) \sinh n(\pi-y) - \sinh n(\pi-x) \sin n(\pi-y) \Big] ,
\end{align*}
as we now explain. Observe that $-\Delta u = 1$ because the infinite series is a harmonic function, and $u=0$ on the boundary of $T$ by simple calculations with Fourier series when $0<x<\pi,y=0$, and when $x=\pi,0<y<\pi$; also $u=0$ on the hypotenuse where $y=x$. 

The torsion function is known to attain its maximum somewhere on the line of symmetry $y=\pi-x$, either by  general symmetry results \cite{CD11,CDK14} or else by arguing as in the proof of \autoref{pr:half-disk} part (ii). On that line of symmetry we evaluate
\begin{align*}
& u(x,\pi-x) \\
& = - (x-\pi/2)^2 + \sum_{n=1}^\infty \frac{n^2 \pi^2 -2\big( 1 - (-1)^n \big)}{\pi n^3 \sinh n\pi} \big[ (-1)^{n+1} \sinh nx - \sinh n(\pi-x) \big] \sin nx .
\end{align*}
The series converges exponentially on each closed subinterval of $(0,\pi)$, and so we may differentiate term-by-term to find 
\begin{align}
& \frac{d\ }{dx} u(x,\pi-x) \label{eq:derivseries} \\
& = - 2(x-\pi/2) + \sum_{n=1}^\infty \frac{n^2 \pi^2 -2\big( 1 - (-1)^n \big)}{\pi n^2 \sinh n\pi} \Big\{ \big[ (-1)^{n+1} \cosh nx + \cosh n(\pi-x) \big] \sin nx \notag \\
& \qquad \qquad \qquad \qquad \qquad \qquad \qquad \qquad \qquad + \big[ (-1)^{n+1} \sinh nx - \sinh n(\pi-x) \big] \cos nx \Big\} , \notag
\end{align}
where once again the series converges exponentially on closed subintervals of $(0,\pi)$. 

The absolute value of the $n$-th term in series \eqref{eq:derivseries} is bounded  by 
\[
\frac{\pi(e^{nx}+e^{n(\pi-x)})}{\sinh(n\pi)} < 3\pi (e^{-n(\pi-x)}+e^{-nx}) ,
\] 
as we see by bounding the $\sin$ and $\cos$ terms with $1$, adding the $\sinh$ and $\cosh$ terms having the same arguments, and using that $\sinh(n \pi) > e^{n\pi}/3$ for $n \geq 1$. Hence the infinite series \eqref{eq:derivseries} is bounded term-by-term by $3\pi$ times the sum of two geometric series having ratios $e^{-(\pi-x)}$ and $e^{-x}$. 

The derivative of $u$ along the line of symmetry is positive at $x=2.1860525$ and negative at $x=2.1860530$, as one finds by evaluating the first 20 terms of the series in \eqref{eq:derivseries} and then estimating the remainder with the geometric series as above. Hence $u$ has a local maximum at $x=2.186053$ to 6 decimal places. This local maximum is a global maximum because $\sqrt{u}$ is concave (see \cite[Example 1.1]{B85} or \cite{K84}). Translating to the left and downwards by $\pi/2$ and then scaling down by a factor of $\sqrt{2}/\pi$ and rotating counterclockwise by $45$ degrees, we find the torsion function on the original triangle has a  maximum at 
\[
x= \frac{2}{\pi} (2.186053-\pi/2) = 0.39168
\]
to 5 decimal places. 
\end{proof}

\section{\bf Concluding remarks}
\label{sec:remarks}

The counterexamples in this paper concern Poisson's equation for $f(z)=1$ and $f(z)=\lambda z$. One can find a whole family of counterexamples using $f(z)=a+bz$, where $a>0$ and $0<b \leq \lambda$. Note the maximum point depends on $b$ but not $a$, as one checks by rescaling the solution $u$ to $u/a$. To study this maximum point as $b$ varies, one starts with the eigenfunctions of $-\Delta-b$ on the half-disk or right isosceles triangle and notes that the eigenfunctions are the same as for $-\Delta$, just with eigenvalues shifted by $b$. The corresponding torsion function can be computed in terms of an eigenfunction expansion, and then the position of the maximum point can be carefully numerically located. We leave such investigations to the interested reader. 

Finally, while our counterexamples involve linear Poisson equations, our choices of $f$ could presumably be perturbed to obtain genuinely nonlinear counterexamples.

\section*{Acknowledgments}
This work was partially supported by grants from the Simons Foundation (\#204296 to Richard Laugesen) and Polish National Science Centre (2012/07/B/ST1/03356 to Bart\-{\l}omiej Siudeja). We are grateful to the Institute for Computational and Experimental Research in Mathematics (ICERM) for supporting participation by Benson and Minion in IdeaLab 2014, where the project began. Thanks go also to the Banff International Research Station for supporting participation by Laugesen and Siudeja in the workshop ``Laplacians and Heat Kernels: Theory and Applications (March 2015), during which some of the research was conducted.

\newcommand{\doi}[1]{%
 \href{http://dx.doi.org/#1}{doi:#1}}
\newcommand{\arxiv}[1]{%
 \href{http://front.math.ucdavis.edu/#1}{ArXiv:#1}}
\newcommand{\mref}[1]{%
\href{http://www.ams.org/mathscinet-getitem?mr=#1}{#1}}

\end{document}